\documentclass[12pt]{article}
\usepackage{graphicx, amsmath, amssymb, amsthm,anysize,float,framed,paralist,booktabs}
\usepackage{todonotes}
\usepackage{algorithm,algorithmic}
\usepackage{hyperref}
\hypersetup{citecolor=red, linkcolor=blue, colorlinks=true}

\marginsize{2cm}{2cm}{1cm}{3cm}

\newcommand{\scrP}{\mathcal{P}}
\newcommand{\scrL}{\mathcal{L}}
\newcommand{\scrH}{\mathcal{H}}
\newcommand{\scrO}{\mathcal{O}}

\newcommand{\PG}{\text{PG}}
\newcommand{\bbF}{\mathbb{F}}

\DeclareMathOperator\h{\mathsf{H}}

\DeclareMathOperator\sfI{\mathrm{I}}
\DeclareMathOperator\go{\mathsf{G}\mathsf{O}}

\newcommand{\dist}{\mathrm{d}}

 \numberwithin{equation}{section}
 \newtheorem{theorem}[equation]{Theorem}
 \newtheorem{cor}[equation]{Corollary}
 \newtheorem{lemma}[equation]{Lemma}
 \newtheorem{prop}[equation]{Proposition}
 \theoremstyle{definition}
 \newtheorem{definition}[equation]{Definition}
 \theoremstyle{remark}
 \newtheorem{rem}[equation]{Remark}

\title
 {Some non-existence results for distance-$j$ ovoids in small generalized polygons}
\author{Anurag Bishnoi and Ferdinand Ihringer}

\begin{document}

\maketitle

\begin{abstract}
  We give a computer-based proof for the non-existence of distance-$2$ ovoids in the dual split Cayley hexagon $\h(4)^D$. 
  Furthermore, we give upper bounds on partial distance-$2$ ovoids of $\h(q)^D$ for $q \in \{2, 4\}$.
\end{abstract}

\section{Introduction}

The study of distance-$j$ ovoids in generalized polygons was started by Thas,
who investigated the existence of distance-$2$ ovoids in generalized quadrangles and distance-$3$ ovoids in generalized hexagons (which are simply known as ovoids) \cite{Thas1981}.
The existence of distance-$j$ ovoids is related to the existence of particular perfect codes \cite{Cameron1976},
the separability of particular groups \cite{Cameron2008}, and various other topics.

The focus of this work is on distance-$2$ ovoids in the \emph{dual split Cayley hexagon} $\h(q)^D$.
While for the \emph{split Cayley hexagon} $\h(q)$ itself the existence of distance-$2$ ovoids is already known for $q=2,3,4$ \cite{DeWispelaere2004, DeWispelaere-VanMaldeghem2005, DeWispelaere2008} and the ovoids completely classified \cite[Sec.~18.3]{Pech2009}, for $\h(q)^D$ only the non-existence for $q=2$ \cite{Frohardt1994} and the existence for $q=3$ \cite{DeWispelaere2004} is known.
Note that we have $\h(q)$ isomorphic to $\h(q)^D$ if and only if $q$ is a power of $3$ \cite[Cor.~3.5.7]{vanMaldeghem1998}. 
Here we present a computer-based proof for the next open case, $\h(4)^D$.

\begin{theorem}\label{thm:non_ex_dualsplitcayley}
  The dual split Cayley hexagon $\h(q)^D$ does not possess a distance-$2$ ovoid for $q \in \{2, 4\}$.
\end{theorem}

The proof uses a combination of various algorithmic ideas, mostly Knuth's dancing links algorithm \cite{Knuth2000}, Linton's smallest image set algorithm \cite{Linton2004} and integer linear programming.
We note that Theorem \ref{thm:non_ex_dualsplitcayley} has been used in \cite{Bishnoi2016} to prove that there does not exist any semi-finite generalized hexagon containing $\h(4)^D$ as a full subgeometry.
In fact, non-existence of distance-$2$ ovoids in any given finite generalized hexagon implies that every generalized hexagon containing the given hexagon as a full subgeometry is finite \cite[Cor.~3.7]{Bishnoi2016}.

It was shown in \cite{Offer2005} that a distance-$3$ ovoid in  a generalized octagon of order $(s, t)$ can only exist if $s = 2t$. 
This implies the non-existence of distance-$3$ ovoids of the dual Ree-Tits octagon $\go(q^2, q)$ for all $q > 2$. 
Computationally, we show that the last remaining case, $\go(4, 2)$, does not possess a distance-$3$ ovoid.
A different computation to verify this fact was already done by Brouwer \cite{Brouwer2011}.
Brouwer's result is mentioned as a remark in a liber amicorum in Dutch, which leaves out some details of the used techniques, and it is not 
connected to the result by Offer and van Maldeghem in his remark; so it seems worthwhile to restate their combined results as follows.

\begin{theorem}[Brouwer, Offer, van Maldeghem]\label{thm:non_ex_octagon}
  The dual Ree-Tits octagon $\go(q^2, q)$ does not possess a distance-$3$ ovoid for all prime powers $q$.
\end{theorem}

\section{Preliminaries}

\subsection{Generalized Polygons}

  A \emph{point-line geometry} is a triple $(\scrP, \scrL, \sfI)$, $\scrP$ and $\scrL$ disjoint, 
  $\sfI \subseteq \scrP \times \scrL$.
  The elements of $\scrP$ are called \emph{points}, the elements of $\scrL$ are called \emph{lines},
  the relation $\sfI$ is called \emph{incidence relation}.
  The \textit{point-line dual }of the geometry $(\scrP, \scrL, \sfI)$ is the geometry $(\scrP^D, \scrL^D, \sfI^D)$ where $\scrP^D = \scrL,  \scrL^D = \scrP$ and $(\ell, x) \in \sfI^D$ iff $(x, \ell) \in \sfI$.
    An \textit{automorphism} of a point-line geometry $(\scrP, \scrL, \sfI)$ is a bijective map $f : \scrP \cup \scrL \rightarrow \scrP \cup \scrL$ such that $f(\scrP) = \scrP$, $f(\scrL) = \scrL$ and $(x, \ell) \in \sfI$ if and only if $(f(x), f(\ell)) \in \sfI$. 
  The \emph{incidence graph} of a point-line geoemtry $(\scrP, \scrL, \sfI)$ has $\scrP \cup \scrL$ as
  its vertices and two vertices are adjacent if and only if they are incident.
  We denote the distance function in this graph by $\delta(\cdot, \cdot)$.
  The \emph{point graph} of a point line geometry $(\scrP, \scrL, \sfI)$ has $\scrP$ as
  its vertices and two vertices are adjacent if they have distance $2$ in the incidence graph, i.e., they are collinear with a common line. 
  We usually denote the point graph by $\Gamma$ and denote its distance function by $\dist(\cdot, \cdot)$.
  A point-line geometry is \textit{connected} if its incidence graph, or equivalently its point graph, is connected. 
  For a point $x$ and a line $\ell$ we define $\dist(x, \ell) := \min \{\dist(x, y) : y ~\sfI~\ell\}$.
  Similarly for two lines $\ell_1, \ell_2$ we define $\dist(\ell_1, \ell_2) = \min\{\dist(x,y) : x ~\sfI~\ell_1, y~\sfI~\ell_2\}.$
  The set of points at distance at most $i$ from a point $x$ in the point graph will be denoted by $\Gamma_{\leq i}(x)$ and the set set of points at distance at most $i$ from a line $\ell$ will be denoted by $\Gamma_{\leq i}(\ell)$.
  The following lemma relates the distance function $\delta$ to the distance function $\dist$. We leave its proof to the reader.

  \begin{lemma}
  \label{lem:dist_delta}
  Let $(\scrP, \scrL, \sfI)$ be a connected point-line geometry, let $\delta(\cdot, \cdot)$ denote distance function in its incidence graph, and let $\dist(\cdot, \cdot)$ denote the distance function in its point graph. 
  Let $x, y \in \scrP$ and $\ell, \ell' \in \scrL$ with $\ell \neq \ell'$.
  Then we have $\delta(x, y) = 2\dist(x, y)$, $\delta(x, \ell) = 2\dist(x, \ell) + 1$ and $\delta(\ell, \ell') = 2\dist(\ell, \ell') + 2$. 
  \end{lemma}
  
  A \emph{generalized $n$-gon} ($n \geq 2$) of order $(s, t)$ is a point-line geometry $(\scrP, \scrL, \sfI)$, $\scrP$ non-empty, such that
  \begin{enumerate}[(a)]
   \item each $\ell \in \scrL$ is incident with $s+1$ elements of $\scrP$,
   \item each $x \in \scrP$ is incident with $t+1$ elements of $\scrL$,
   \item the incidence graph has diameter $n$ and the maximum possible girth, $2n$.
  \end{enumerate}

By a famous result of Feit and Higman \cite{Feit-Higman}, generalized $n$-gons of order $(s, t)$ with $s, t > 1$ (the \textit{thick} case) exist only for $n \in \{ 2, 3, 4, 6, 8 \}$. 
 For $n = 2$ we have a geometry $(\scrP, \scrL, \sfI)$ where $\sfI = \scrP \times \scrL$ and for $n = 3$ we have a finite projective plane. 
Generalized $n$-gons for $n = 4$, $6$ and $8$ are referred to as generalized \textit{quadrangles}, \textit{hexagons} and \textit{octagons}, respectively. 
By an easy counting, the number of points in a generalized hexagon of order $(s, t)$ is $(1 + s)(1 + st + s^2 t^2)$ and the number of points in a generalized octagon of order $(s, t)$ is $(1 + s)(1 + st + s^2t^2 + s^3t^3)$. 
From the axioms of a generalized polygon it follows that the point-line dual of a generalized polygon of order $(s, t)$ is a generalized polygon of order $(t, s)$. 

For $n = 2d$, axiom (c) in the definition of generalized $n$-gons can be replaced by the following set of axioms on the point graph of the geometry \cite[Sec. 1.9.4]{DeBruyn2006_book}: 
\begin{enumerate}[(1)]
\item For every line $\ell$ and every point $x$ there exists a unique point $x'$ on $\ell$ such that $\dist(x, y) = \dist(x, x') + 1$ for all $y \neq x'$ on $\ell$. 
\item For every two points $x$, $y$ with $\dist(x, y) = i < d$ there exists a unique neighbour of $y$ in the point graph which is at distance $i - 1$ from $x$. 
\end{enumerate}

We denote the \emph{Desarguesian projective plane} over $\bbF_q$ by $\PG(2, q)$.
Then $\h(q, 1)$ denotes the generalized hexagon of order $(q, 1)$ whose points are the incident point-line pairs of $\PG(2, q)$, lines are the points and lines of $\PG(2, q)$, and incidence is reverse containment.

Let $\ell$ be a $2$-dimensional subspace of $\bbF_q^n$, where $q$ a prime power and $n \geq 2$.
Let $x  = (x_1, \dots, x_n), y = (y_1, \dots, y_n)$ be a basis of $\ell$. Then the Grassmann coordinates of $\ell$
are $(x_iy_j - x_jy_i)_{1 \leq i < j \leq n}$. Notice that the Grassmann coordinates are 
independent of the choice of $x$ and $y$, up to scalar multiplication.
The \emph{dual split Cayley hexagon} $\h(q)^D$ is a generalized hexagon of order $(q, q)$ and can be defined as follows \cite[Chap. 2]{Tits1959, vanMaldeghem1998}.
Define the quadratic form $Q: \bbF_q^7 \rightarrow \bbF$ with $Q(x) = x_0x_4 + x_1x_5 + x_2x_6 - x_3^2$.
\begin{enumerate}[(a)]
 \item The \textit{lines} of $\h(q)^D$ all $1$-dimensional subspaces of $\bbF_q^7$, which vanish on $Q$.
 \item The \textit{points} of $\h(q)^D$ are all $2$-dimensional subspaces of $\bbF_q^7$, which vanish on $Q$
	and whose Grassmann coordinates satisfy $p_{12} = p_{34}$, $p_{54} = p_{32}$, $p_{20} = p_{35}$,
	$p_{65} = p_{30}$, $p_{01} = p_{36}$ and $p_{46} = p_{31}$.
 \item Incidence is reverse containment. 
\end{enumerate}

Let $q = p^r$, where $p$ is a prime and $r$ is a positive integer. 
Then the automorphism group of $\h(q, 1)$ is isomorphic to $\mathrm{P\Gamma L}_3(q) \rtimes C_2$ and thus it has size $2r(q^3 - 1)(q^3 - q)(q^3  - q^2)/(q - 1)$.
The automorphism group of $\h(q)$ is isomorphic to $\mathrm{G}_2(q) \rtimes \mathrm{Aut}(\mathbb{F}_q)$ and thus it has size $rq^6(q^6 - 1)(q^2 - 1)$. 
The following is a well known result on the relationship between these generalized hexagons. 

\begin{lemma}[{\cite[Cor. 1.8.6]{vanMaldeghem1998}}]
\label{lem:subhexagon}
  The dual split Cayley hexagon $\h(q)^D$ contains a subhexagon $\scrH$ of order $(q, 1)$ ismorphic to $\h(q, 1)$.
  Moreover, for every pair of lines $\ell_1, \ell_2  \in \h(q)^D$ which are at distance $6$ from each other in the incidence graph there is a unique $\h(q, 1)$-subhexagon of $\h(q)^D$ which contains both $\ell_1$ and $\ell_2$. 
\end{lemma}
\begin{cor}
\label{cor:subhexagon}
The number of subhexagons of $\h(q)^D$ that are isomorphic to $\h(q, 1)$ is equal to $q^3(1+q)(q^2 - q + 1)/2$.
\end{cor}
\begin{proof}
Let $\delta(\cdot, \cdot)$ denote the distance function in the incidence graph of $\h(q)^D$. Double count the triples $(\ell_1, \ell_2, \scrH)$ where $\ell_2, \ell_2$ are two lines of $\h(q)^D$ with $\delta(\ell_1, \ell_2) = 6$ and $\scrH$ is a subhexagon isomorphic to $\h(q, 1)$ that contains both $\ell_1$ and $\ell_2$. 
There are in total $(1+q)(1 + q^2 + q^4)$  lines in $\h(q)^D$ and $q^5$ lines at distance $6$ from a fixed line. 
Therefore, there are $q^5(1 + q)(1+q^2 + q^4)$ such triples. 
There are in total $2(1+q+q^2)$ lines in $\h(q, 1)$ and $q^2$ lines at distance $6$ from a fixed line. 
Thus, if $k$ is the total number of subhexagons isomorphic to $\h(q, 1)$, then we have $kq^2(2 + 2q + 2q^2) = q^5(1 + q)(1 + q^2 + q^4)$, which gives us $k = q^3(1+q)(q^2 - q + 1)/2$. 
\end{proof}

The \emph{dual Ree-Tits octagon} $\go(q^2, q)$, $q$ an odd power of $2$, 
is a generalized octagon of order $(q^2, q)$ and its definition can be seen in \cite{Tits1983} or \cite{Coolsaet2005}.

\subsection{Ovoids and Associated Algorithms}

The usual definition of a distance-$j$ ovoid, $j \geq 1$, of a generalized polygon is the following \cite{Offer2005}.

\begin{definition}
  Let $\mathcal{S} = (\scrP, \scrL, \sfI)$ be a generalized $2d$-gon and let $2 \leq j \leq d$.
  \begin{enumerate}[(a)]
    \item A \emph{partial distance-$j$ ovoid} of $\mathcal{S}$ is a set of points $\scrO$ such that all elements of $\scrO$ have distance at least $2j$ (in the incidence graph) from each other. 
    \item A \emph{distance-$j$ ovoid} of $\mathcal{S}$ is a partial distance-$j$ ovoid $\scrO$ such that every element of $\scrP \cup \scrL$ has distance at most $j$ from at least one element of $\scrO$.
  \end{enumerate}
\end{definition}

As a consequence of Lemma \ref{lem:dist_delta} we have the following equivalent definition \cite[Sec.~3.5]{DeBruyn2005_valuations} in terms of the point-graph which we will use in this paper. 

\begin{definition}
Let $\mathcal{S} = (\scrP, \scrL, \sfI)$ be a generalized polygon and let $\dist(\cdot, \cdot)$ denote the distance function in the point graph of $\mathcal{S}$. 
\begin{enumerate}[(a)]
\item A \emph{partial distance-$j$ ovoid} of $\mathcal{S}$ is a set of points $\scrO$ such that for every two distinct points $x$ and $y$ we have $\dist(x, y) \geq j$. 

\item A \emph{distance-$j$ ovoid} of $\mathcal{S}$ is a partial distance-$j$ ovoid $\scrO$ such that (1) for every point $a$ of $\mathcal{S}$ there exists a point $x$ of $\mathcal{O}$ such that $\dist(a, x) \leq j/2$; (2) for every line $\ell$ of $\mathcal{S}$ there exists a point $x \in \mathcal{O}$ such that $\dist(\ell, x) \leq (j - 1)/2$. 
\end{enumerate}
\end{definition}

\begin{lemma}\footnote{One side of this Lemma is proved in \cite[Lem. 2]{DeBruyn2006_ovoids}.}
\label{lem:exact_cover}
Let $\mathcal{S} = (\mathcal{P}, \mathcal{L}, \mathsf{I})$ be a generalized $2d$-gon. 
For any $i \in \{0, \dots, d\}$ and an element $a \in \mathcal{P} \cup \mathcal{L}$, let $\Gamma_{\leq i}(a)$ denote the set of points at distance at most $i$ from $a$ in the point graph of $\mathcal{S}$.
Let $\mathcal{O}$ be a set of points and $j \in \{2, \dots, d\}$. 
Then 
\begin{enumerate}[$(1)$]
\item for $j$ even, $\mathcal{O}$ is a distance-$j$ ovoid if and only if for all $\ell \in \mathcal{L}$ we have $|\Gamma_{\leq (j - 2)/2}(\ell) \cap \mathcal{O}| = 1$. 

\item for $j$ odd, $\mathcal{O}$ is a distance-$j$ ovoid if and only if for all $x \in \mathcal{P}$ we have $|\Gamma_{\leq (j - 1)/2}(x) \cap \mathcal{O}|  = 1$. 
\end{enumerate}
\end{lemma}
\begin{proof}
We only prove the first case, when $j$ is even, and note that the second part has a similar proof. 
Say $\mathcal{O}$ is a distance-$j$ ovoid and let $\ell \in \mathcal{L}$. 
Then by the definition of distance-$j$ ovoids there exists a point $x$ in $\mathcal{O}$ such that $\dist(x, \ell) \leq (j - 1)/2$, but since $j$ is even and distances are integral we have $\dist(x, \ell) \leq (j - 2)/2$. 
Say there was another point $y \neq x$ in $\mathcal{O}$ with $\dist(y, \ell) \leq (j - 2)/2$. 
Then $\dist(x, y) \leq \dist(x, \ell) + \dist(y, \ell) + 1 = j - 1$ which is a contradiction. 
Now say $\mathcal{O}$ is a set of points such that for every line $\ell$ we have $|\Gamma_{\leq (j - 2)/2}(\ell) \cap \mathcal{O}| = 1$. 
Let $x, y$ be two distinct points in $\mathcal{O}$. 
If $\dist(x, y) \leq j - 1$, then there exits a line $\ell$ in the path joining $x$ to $y$ such $\dist(x, \ell) \leq (j -2)/2$ and $\dist(y, \ell) \leq (j - 2)/2$, which is not possible. 
Now let $x$ be an arbitrary point of $\mathcal{S}$. 
Let $\ell$ be any line through $x$, and let $y$ be the unique point in $\mathcal{O}$ such that $\dist(\ell, y) \leq (j - 2)/2$. 
Then $\dist(x, y) \leq 1 + \dist(\ell, y) = j/2$. 
Let $\ell$ be an arbitrary line of $\mathcal{S}$, then by the assumption on $\mathcal{O}$ there exists a point in $\mathcal{O}$ at distance at most $(j - 2)/2 \leq (j - 1)/2$ from $\ell$. 
Therefore, $\mathcal{O}$ is a distance-$j$ ovoid. 
\end{proof}

The \textit{exact cover} problem in a hypergraph $(V, E)$ asks for the existence of a subset $S$ of $E$ such that for every vertex $v$ there exists a unique edge $e$ in $S$ which contains $v$. 
The dual of this problem is the \textit{exact hitting set} problem where we need to find a subset $O$ of $V$ such that for every edge $e$ there is a unique vertex $v$ in $O$ which is contained in $E$. 
Lemma \ref{lem:exact_cover} makes it clear that the existence of a distance-$j$ ovoid in a generalized $2d$-gon $\mathcal{S}$ is equivalent to existence of an \textit{exact hitting set} in a hypergraph derived from the point graph of $\mathcal{S}$. 
For $j$ even the edges of this hypergraph are the subsets $\Gamma_{\leq (j-2)/2}(\ell)$ of $\mathcal{P}$ where $\ell$ is a line, and for $j$ odd the edges of this hypergraph are the subsets $\Gamma_{\leq (j - 1)/2}(x)$ where $x$ is a point. 
This makes it possible to use Knuth's \emph{dancing links algorithm for exact covers} \cite{Knuth2000} to find all distance-$j$ ovoids. 
Note that the exact cover problem is NP-hard. 

A second technique which is available for the exact cover problems is the use of integer linear programming solvers.
We will use it in the following way.
Let $\mathcal{S} = (\mathcal{P}, \mathcal{L}, \mathrm{I})$ be a generalized $2d$-gon. 
Let $\scrO'$ be a possibly empty set of points which forms a partial distance-$j$ ovoid, i.e., every pair of points in $\scrO'$ are at distance at least $j$ in the point graph. 
Let $H = (V, E)$ be the hypergraph as defined above, with $V = \mathcal{P}$ and 
\[E = 
\begin{cases}
\{\Gamma_{\leq (j - 2)/2}(\ell) : \ell \in \mathcal{L}\} \text{ if } j \text { is even}\\
\{\Gamma_{\leq (j - 1)/2}(p) : p \in \mathcal{P}\} \text{ if } j \text{ is odd}.
\end{cases}
\]
For each $p \in \scrP$ let $X_{p} \in \{ 0, 1\}$ be a binary variable.
Then the equations
\begin{align}
  X_p = 1 && \text{ for all } p \in \scrO' \notag\\
  \sum_{p \in e} X_p = 1 && \text{ for all } e \in E \label{eq:MIP_for_ovoid}
\end{align}
have an integer solution if and only if $\mathcal{S}$ possesses a distance-$j$ ovoid that contains $\scrO'$.
Similarly, the equations
\begin{align}
  X_p = 1 && \text{ for all } p \in \scrO' \notag\\
  \sum_{p \in e} X_p \leq 1 && \text{ for all } e \in E \label{eq:MIP_for_partil_ovoid}
\end{align}
have an integer solution is and only if $\mathcal{S}$ possesses a partial distance-$j$ ovoid that contains $\scrO'$.

Any of these formulations can be directly used to prove Theorem \ref{thm:non_ex_dualsplitcayley} for $q = 2$ and Theorem \ref{thm:non_ex_octagon}. 
We have verified this using Gurobi\footnote{The running time was about one day with Gurobi Optimizer version 6.5.0 build v6.5.0rc1 (linux64) with an Intel Core i5-3550 CPU @ 3.30GHz processor}. 
As noted before, non-existence of distance-$3$ ovoids in $\go(q^2, q)$ for $q > 2$ is covered in \cite{Offer2005}
and the case $q=2$ was already mentioned in \cite{Brouwer2011}.

\section{Distance-$2$ Ovoids in $\h(4)^D$}

\begin{lemma}
\label{lem:dual_split_Cayley}
  Let $\scrH$ be a hexagon of order $(s, t)$.
  Let $\scrH'$ be a subhexagon of order $(s, t')$ of $\scrH$ and
  let $\scrO$ be a distance-$2$ ovoid of $\scrH$. Then
  $\scrH' \cap \scrO$ is a distance-$2$ ovoid of $\scrH'$ and
  \begin{align*}
    |\scrO \cap \scrH'| = s^2t'^2 + st' + 1.
  \end{align*}
\end{lemma}
\begin{proof}
By Lemma \ref{lem:exact_cover}, $\mathcal{O}$ is a distance-$2$ ovoid if and only if it meets every line in a unique point.  
 If each line of $\scrH$ meets $\scrO$ in exactly $1$ points, then the same is true for $\scrH'$.
 Moreover, the number of points in a generalized hexagon of order $(s, t)$ is $(1 + s)(1 + st + s^2t^2)$, and thus by double counting, the number of points in a distance-$2$ ovoid is $(1 + st + s^2t^2)$. 
 Therefore, we have $|\scrO \cap \scrH'| = s^2t^2 + st' + 1$. 
\end{proof}

While both Knuth's dancing links algorithm and integer programming solvers fail to directly determine the existence distance $2$-ovoids in $\h(4)^D$ which has $1365$ points and $1365$ lines, in view of Lemmas \ref{lem:subhexagon} and \ref{lem:dual_split_Cayley} we can use the following idea: 
\emph{first classify all distance-$2$ ovoids in $\h(4, 1)$ up to isomorphism under the action of the stabilizer of $\h(4, 1)$, and then see if any of these ovoids can be extended to a distance-$2$ ovoid of $\h(4)^D$}. 

\bigskip \noindent
We note that the stabilizer of a subgeometry of $\h(4)^D$ which is isomorphic to $\h(4, 1)$, under the action of the automorphism group of $\h(4)^D$ is in fact isomorphic to the automorphism group of $\h(4, 1)$.
As the point graph of $\h(q, 1)^D$ corresponds to the incidence graph of the projective plane $\PG(2, q)$, a distance-$2$ ovoid in $\h(q, 1)$ corresponds to a perfect matching of the incidence graph of $\PG(2, q)$. 
It is folklore that the number of perfect matchings in a balanced bipartite graph corresponds to the permanent of the biadjacency matrix of that graph (see for example \cite{Plummer2015}).
It is easy to verify the following by calculating the corresponding permanent.

\begin{lemma}[{\cite[A000794]{OEIS}}]\label{lem:permanent_h41}
  The number of perfect matchings in the incidence graph of $\mathrm{PG}(2, 4)$ is $18534400$.
\end{lemma}

Notice that a perfect matching is an exact cover, and so we can use Knuth's dancing links algorithm to enumerate all perfect matchings in a bipartite graph.

\begin{prop}\label{prop:class_h4d_ovoids}
  Let $G$ be the automorphism group of $\h(4)^D$.
  Let $\scrH$ be a subhexagon of $\h(4)^D$ ismormorphic to $\h(4, 1)$.
  Then there are exactly $350$ non-isomorphic distance-$2$ ovoids in $\scrH$ with respect to $G_{\scrH}$, the stabilizer of $\scrH$ under the action of $G$.
\end{prop}

We used a computer to prove Proposition \ref{prop:class_h4d_ovoids}.
The following algorithm was able to classify all $350$ in a few minutes at the time of writing.\footnote{Running time: 28m37.576s with Sage Version 6.4.1 with a Intel Core i5-2400 CPU @ 3.10 GHz processor. We have to point out that Knuth's dancing link algorithm is partially randomized, the running times might vary for many reasons. 
A different model of the hexagon with the same hardware and the same Sage version has an average running time of circa 120 minutes. The same model with a different Sage version on a slower processor has a average running time of circa 15 minutes.}
We rely on Linton's algorithm \texttt{SmallestImageSet(H, S)}, which returns the lexicographically
smallest element in the orbit of a set \texttt{S} under the action of a group \texttt{H} \cite{Linton2004}.

% \begin{algorithm}
% \caption{Classification of distance-$2$ ovoids of $\h(4, 1)$.}
\label{alg:class_ovoids_h_4_1}
\begin{algorithmic}
  \STATE Let $i$ be an iterator on all distance-$2$ ovoids of $\h(4, 1)$.
  \STATE $b \leftarrow 18534400$
  \STATE $L \leftarrow \{ \}$
  \WHILE{$b > 0$}
    \STATE $m \leftarrow i.\text{next}$
    \STATE $m \leftarrow $ \texttt{SmallestImageSet($H$, $m$)}
    \IF{$m \notin L$}
      \STATE $L \leftarrow L \cup \{ m \}$
      \STATE $s \leftarrow $ the orbit length of $m$ under $G_{\scrH}$
      \STATE $b \leftarrow b - s$
    \ENDIF
  \ENDWHILE
\end{algorithmic}
% \end{algorithm}

After running the algorithm, $L$ contains all distance-$2$ ovoids of $\h(4, 1)$.
We used the implementation of Dancing Links in SAGE \cite{sage} \footnote{\url{http://www.sagenb.org/src/combinat/matrices/dlxcpp.py}} to find the iterator and the implementation of \texttt{SmallestImageSet} in the GRAPE \cite{grape} package of GAP \cite{GAP4} to find the representatives of these $350$ isomorphism classes of distance-$2$ ovoids. We provide a more explicit description of these $350$ distance-$2$ ovoids at the end of this section.
We provide a list of all non-isomorphic 350 distance-$2$ ovoids and our full code online.\footnote{\url{http://math.ihringer.org/data.php}}
For each distance-$2$ ovoid $\scrO'$ of $\scrH$ we can define a integer linear optimization problem (ILP)
as in \eqref{eq:MIP_for_ovoid}. Then the ILP solvers easily shows that these equations are infeasible for all of the $350$ cases.%
\footnote{We verified this with CPLEX (several versions), Gurobi Optimizer (several versions) and the constraint solver Minion. The 350 ILPs in 350 files in the LP format 
took 540.3 seconds with Gurobi Optimizer version 6.5.0 build v6.5.0rc1 (linux64) with an Intel Core i5-3550 CPU @ 3.30GHz processor. Minion's running times were similar.} 
This proves Theorem \ref{thm:non_ex_dualsplitcayley}.

\begin{rem}
  For the next open case, $\h(5)^D$, our algorithmic approach fails for several reasons:
  \begin{enumerate}[(a)]
   \item The incidence graph of $\PG(2, 5)$ has $4598378639550$ perfect matchings while the automorphism group of $\PG(2, 5)$ has size $744000$.
   So a classification of all non-isomorphic distance-$2$ ovoids of $\h(5, 1)$ seems to be out of reach.
   \item Even for one given distance-$2$ ovoid of $\h(5, 1)$, the corresponding integer linear program takes too long to solve with state-of-the-art ILP solver.
  \end{enumerate}
\end{rem}

One can use the same methods to obtain bounds on partial distance-$2$ ovoids.

\begin{lemma}\label{lem:bnd_without_full_subhex}
  Let $\scrO$ be a partial distance-$2$ ovoid of $\h(q)^D$. Suppose that no subhexagon $\scrH$ of $\h(q)^D$
  isomorphic to $\h(q, 1)$ contains $q^2+q+1$ points of $\scrO$. Then $|\scrO| \leq (q^2-q+1) (q^2+q)$
\end{lemma}
\begin{proof}
  Let $\scrP$ be the  set of points of $\h(q)^D$.
 We double count $(p, \scrH)$, where
  $\scrH$ a subhexagon of $\h(q)^D$ isomorphic to $\h(q, 1)$ and $p \in \mathcal{O} \cap \scrH$.
From a counting argument similar to the one in the proof of Corollary \ref{cor:subhexagon}, we see that each point is contained in $(1+q)q^3/2$ subhexagons isomorphic to $\h(q, 1)$ which tells us that there are $|\mathcal{O}|(1+q)q^3/2$ such pairs. 
 Again by Corollary \ref{cor:subhexagon}, there are $q^3(1+q)(q^2 - q + 1)/2$  subhexagons of $\h(q)^D$ which are isomorphic to $\h(q, 1)$. 
  Under the condition $|\scrO \cap \scrH| \leq q^2+q$ this yields
    $|\scrO| \leq (q^2-q+1) \cdot (q^2+q)$.
\end{proof}

For $q = 2$, Lemma \ref{lem:bnd_without_full_subhex} gives us $|\mathcal{O}| \leq 18$ and for $q = 4$ it gives us $|\mathcal{O}| \leq 260$ under the given assumptions.
To prove that the bounds given by Lemma \ref{lem:bnd_without_full_subhex} hold \textit{for all} partial distance-$2$ ovoids of $\h(q)^D$, $q \in \{2, 4\}$, we can use the following computational approach. 
If the ILP defined in \eqref{eq:MIP_for_partil_ovoid} does not have a solution larger than some integer $b \geq (q^2-q+1)(q^2+q)$ for all of the $350$ non-isomorphic distance-$2$ ovoids of $\h(q, 1)$, then we obtain $b$ as an upper bound on the size of a partial distance-$2$ ovoids.
We are able to obtain the following results using this approach.

\begin{lemma}
  A partial distance-$2$ ovoid $\scrO$ of $\h(q)^D$ satisfies the following:
  \begin{enumerate}[$(a)$]
   \item $|\scrO| \leq 19$ for $q=2$. 
   \item $|\scrO| \leq 265$ for $q=4$.
  \end{enumerate}
\end{lemma}

In fact, one can easily construct a partial distance-$2$ ovoid of size $19$ in $\h(2)^D$ using a computer. 
So the bound for $\h(2)^D$ is sharp. 
With Lemma \ref{lem:bnd_without_full_subhex} the bound we obtain for $\h(4)^D$ is $q^4+q=260$. 
We suspect that
this is the true bound, but testing one of the $350$ partial distance-$2$ ovoids takes about 2 days with our methods,
so we end up with an unreasonable running time of $2$ years.\footnote{The bound $265$ takes circa 
one week with our methods, which is more reasonable to verify.}

We conclude this work by giving a more explicit description of the $350$ perfect matchings of $\PG(2, 4)$.
We provide the structure description of the stabilizers of these ovoids provided by GAP, the lengths of point orbits in $\h(4, 1)$ 
and the lengths of line orbits in $\h(4, 1)$.

\begin{center}
\begin{tabular}{lllll}
Stabilizer Size & Number & Structure & Point Orbit Lengths & Line Orbit Lengths\\ \midrule
126 & 1 & $(C_3 \times{} (C_7 : C_3)) : C_2$ & $42^1 21^1 14^2 7^2$ & $14^3$\\
84 & 4 & $S_3 \times{} D_{14}$ & $28^1 14^4 7^3$ & $28^1 14^1$\\
54 & 1 & $((C_3 \times{} C_3) : C_3) : C_2$ & $18^4 9^1 3^8$ & $18^1 6^4$\\
42 & 4 & $C_3 \times{} D_{14}$ & $14^3 7^9$ & $14^3$\\
36 & 2 & $S_3 \times{} S_3$ & $12^3 6^{10} 3^1 2^2 1^2$ & $12^1 6^4 2^3$\\
18 & 14 & $C_3 \times{} S_3$ & $6^{13} 3^7 2^2 1^2$ & $6^6 2^3$\\
18 & 2 & $C_3 \times{} S_3$ & $6^{16} 3^1 2^2 1^2$ & $6^6 2^3$\\
12 & 14 & $D_{12}$ & $4^{19} 2^{13} 1^3$ & $4^7 2^7$\\
9 & 3 & $C_3 \times{} C_3$ & $3^{33} 1^6$ & $3^{12} 1^6$\\
6 & 2 & $S_3$ & $2^{42} 1^{21}$ & $2^{14} 1^{14}$\\
6 & 43 & $S_3$ & $2^{50} 1^5$ & $2^{21}$\\
6 & 121 & $C_6$ & $2^{48} 1^9$ & $2^{21}$\\
3 & 139 & $C_3$ & $1^{105}$ & $1^{42}$
\end{tabular}
\end{center}

\section{Conclusion}

As the case $\h(5)^D$ is computationally out of reach, 
the next goal should be to replace the computational parts of our proof for $\h(4)^D$ with algebraic arguments.
The investigation of the structure of the $350$ distance-$2$ ovoids of $\h(q, 1)$ shows that it might not be feasible to
describe the these distance-$2$ ovoids explicitly.
Maybe the specific structure of a distance-$2$ ovoid of $\h(q, 1)$ is far less important than the 
fact that all subhexagons of $\h(q)^D$ isomorphic to $\h(q, 1)$ meet a distance-$2$ ovoid in exactly $q^2+q+1$ points.

\bibliographystyle{plain}
%\bibliography{jabref}

\end{document}